\documentclass{amsart}
\usepackage{amssymb,amscd}
\usepackage{color}
\usepackage{url}

\usepackage
{hyperref}

\newtheorem{theorem}{Theorem}
\newtheorem{proposition}[theorem]{Proposition} \newtheorem{corollary}[theorem]{Corollary}
\newtheorem{lemma}[theorem]{Lemma}

\title{A remark on non-commutative $L^p$-spaces}
\author{Shinya Kato and Yoshimichi Ueda}

\address{
Graduate School of Mathematics, Nagoya University, 
Furocho, Chikusaku, Nagoya, 464-8602, Japan
}
\email{
(SK) kato.shinya.k6@s.mail.nagoya-u.ac.jp; 
(YU) ueda@math.nagoya-u.ac.jp
}

\thanks{This work was supported in part by Grant-in-Aid for Scientific Research (B) JP18H01122.}

\date{\today}

\begin{document}

\begin{abstract}
We explicitly describe the Haagerup and the Kosaki non-commutative $L^p$-spaces associated with a tensor product von Neumann algebra $M_1\bar{\otimes}M_2$ in terms of those associated with $M_i$ and usual tensor products of unbounded operators. The descriptions are then shown to be useful in the quantum information theory based on operator algebras. 
\end{abstract}

\maketitle

\allowdisplaybreaks{


\section{Introduction}\label{S1} 

Quantum information theory (QIT for short) can be developed in the infinite-dimensional (even non-type I) setup with the help of operator algebras (such a general framework is necessary for quantum field theory for example), although QIT is usually discussed in the finite-dimensional setup. In the finite-dimensional setup, the primary objects in QIT are density matrices, which no longer make sense in the non-type I setup. However, Haagerup's theory of non-commutative $L^p$-spaces (see \cite{Terp:thesis81}) allows us to have a certain counterpart of density matrices; actually, the so-called Haagerup correspondence $\varphi \mapsto h_\varphi$ (the operator $h_\varphi$ is sometimes denoted by $\varphi$ itself) between the normal functionals and a class of $\tau$-measurable operators gives a correct counterpart of density matrices in the non-type I setting. 

In QIT, tensor products of systems (i.e., systems consisting of independent subsystems) naturally emerges, and hence it is desirable to clarify how Haagerup non-commutative $L^p$-spaces behave under von Neumann algebra tensor products. In the commutative setup, the answer is simply $L^p(\mu_1\otimes\mu_2) = L^p(\mu_1,L^p(\mu_2)) = L^p(\mu_2,L^p(\mu_1))$ with natural identifications 
by utilizing the concept of vector-valued $L^p$-spaces. However, the concept of vector-valued $L^p$-spaces has not been established yet in full generality of non-commutative setting.

The purpose of this short note is to give some descriptions of the Haagerup and the Kosaki non-commutative $L^p$-spaces associated with a tensor product von Neumann algebra; see Theorem \ref{T6} and Corollary \ref{C7}. Those descriptions are indeed rather natural but has not been given so far to the best of our knowledge. We remark that a similar but abstract result based on interpolation method was given by Junge \cite{Junge:CJM04} before. On the other hand, the descriptions we will give depend upon a technology of the so-called Takesaki duality \cite {Takesaki:ActaMath73} and are provided by means of tensor products of unbounded operators. Consequently, our descriptions are really concrete with multiplicativity of norm. An immediate consequence of our descriptions is a natural proof of the additivity of sandwiched R\'{e}nyi divergences in the non-type I setup due to Berta et al.\ \cite{BertaScholzTomamichel:AHP18} and Jen\v{c}ov\'{a} \cite{Jencova:AHP18,Jencova:AHP21}. Remark that the additivity was claimed by Berta et al., in a different approach to non-commutative $L^p$-spaces (see \cite[page 1860]{BertaScholzTomamichel:AHP18}) and also confirmed by Hiai and Mosonyi \cite[equation (3.16)]{HiaiMosonyi:AHP23} in the injective or AFD von Neumann algebra case. 

\medskip\noindent
{\bf Acknowledgement.} We would like to acknowledge Fumio Hiai for giving a series of online lectures entitled `Quantum Analysis and Quantum Information Theory' during Feb.--Mar.\ 2022 delivered via Nagoya University, supported by Grant-in-Aid for Scientific Research (B) JP18H01122. The lectures motivated us to investigate non-commutative $L^p$-spaces associated with tensor product von Neumann algebras. We also thank him for his comments to a draft of this note. 

\section{Preliminaries}\label{S2} 

The basic references of this short note are \cite{Stratila:Book81} (on modular theory), \cite{Terp:thesis81} (on Haagerup non-commutative $L^p$-spaces), \cite{Kosaki:JFA84-1} (on Kosaki non-commutative $L^p$-spaces), but the reader can find concise expositions on those topics in \cite[Appendix A]{Hiai:Book21} and its expansion \cite{Hiai:Book21-2}.

\medskip
Let $M$ be a von Neumann algebra. Choose a faithful semifinite normal weight $\varphi$ on $M$. The \emph{continuous core} of $M$ is the crossed product $\widetilde{M} := M\bar{\rtimes}_{\sigma^{\varphi}}\mathbb{R}$. Let $\theta^M : \mathbb{R} \curvearrowright \widetilde{M}$ be the \emph{dual action}, which is characterized by 
\begin{equation}\label{Eq1}
\theta_s^M\circ\pi_{\varphi} = \pi_{\varphi}, \quad \theta_s^M(\lambda^{\varphi}(t))= e^{-its}\lambda^{\varphi}(t)
\end{equation}
for all $s,t\in\mathbb{R}$, where $\pi_{\varphi} : M \to M\bar{\rtimes}_{\sigma^{\varphi}}\mathbb{R}$ and $\lambda^{\varphi} : \mathbb{R} \to \widetilde{M}$ denote the canonical injective normal $*$-homomorphism from $M$ and the canonical unitary representation of $\mathbb{R}$ into $\widetilde{M}$ that generated by the $\pi_{\varphi}(a)$ and the $\lambda^{\varphi}(t)$ as a von Neumann algebra. In what  follows, we will identify $a = \pi_{\varphi}(a)$ and $M = \pi_{\varphi}(M)$ unless no confusions are possible. 
Note that the covariant relation
\begin{equation}\label{Eq2}
\lambda^{\varphi}(t)a = \sigma^{\varphi}_t(a)\lambda^{\varphi}(t), \qquad t \in \mathbb{R}, \quad a \in M.  
\end{equation}
holds. We remark that $(\widetilde{M},\theta^M)$ is known to be independent of the choice of $\varphi$ up to conjugacy. 

The canonical trace $\tau_M$ on $\widetilde{M}$ is a faithful semifinite normal tracial weight uniquely determined by 
\begin{equation}\label{Eq3}
[D\widetilde{\varphi}:D\tau_M]_t = \lambda^{\varphi}(t), \qquad t \in \mathbb{R}, 
\end{equation} 
where $[D\widetilde{\varphi}:D\tau_M]_t$ is Connes's Radon--Nikodym cocycle of $\widetilde{\varphi}$ with respect to $\tau_M$. 
Here, $\widetilde{\varphi}$ is the \emph{dual weight} of $\varphi$ defined by 
\begin{equation}\label{Eq4}
\widetilde{\varphi} := \widehat{\varphi}\circ T_M, 
\end{equation}
where $\widehat{\varphi}$ is the canonical extension of $\varphi$ to the \emph{extended positive part} $\widehat{M}_+$ (see e.g., \cite[\S11]{Stratila:Book81}) and $T_M : \widetilde{M}_+ \to \widehat{M}_+$ is the operator-valued weight 
\begin{equation}\label{Eq5} 
T_M(a) := \int_\mathbb{R} \theta^M_t(a)\,dt, \qquad a \in \widetilde{M}_+. 
\end{equation} 

\medskip
In what follows, we denote by $s(\psi)$ the \emph{support projection} of a semifinite normal weight $\psi$. We also use Connes's Radon--Nikodym cocycle with general (not necessarily faithful) semifinite normal weight on the left-hand side, see \cite[\S3]{Stratila:Book81}. 

The next lemma immediately follows from the construction of Connes's Radon--Nikodym derivatives (see \cite[\S3]{Stratila:Book81}) together with \cite[\S2.22, Eq.(1)]{Stratila:Book81}. 

\begin{lemma}\label{L1} Let $\psi$ be a semifinite normal weight on $M$ and $\psi'$ be another semifinite normal weight on $M$ such that $s(\psi') = 1 - s(\psi)$. Then $\chi := \psi + \psi'$ is a faithful semifinite normal weight on $M$ and 
\[
[D\psi:D\varphi]_t = s(\psi)[D\chi:D\varphi]_t, \qquad t \in \mathbb{R}.
\] 
\end{lemma} 

The \emph{Haagerup correspondence} $\varphi \mapsto h_\varphi$ is a bijection from the set of all semifinite normal weights on $M$ onto the positive self-adjoint operators $h$ affiliated with $\widetilde{M}$ satisfying that $\theta^M_s(h) = e^{-s}h$ for every $s\in\mathbb{R}$. 

\begin{lemma}\label{L2} 
Let $\psi$ be a semifinite normal weight on $M$ and 
\[
\widetilde{\psi} := \widehat{\psi}\circ T_M
\]
be its dual weight. Then 
\[
[D\widetilde{\psi}:D\tau_M]_t = [D\psi:D\varphi]_t\,\lambda^{\varphi}(t)
\]
holds for every $t \in \mathbb{R}$, and the Haagerup correspondence $h_\psi$ is uniquely determined by
\begin{equation}\label{Eq6}
h_\psi^{it} = [D\psi:D\varphi]_t\,\lambda^{\varphi}(t), \qquad t \in \mathbb{R}, 
\end{equation}
where $h_\psi^{it}$ is the functional calculus $f_t(h_\psi)$ with function 
\[
f_t(\lambda) := \begin{cases} \lambda^{it}=e^{it\log\lambda} & (\lambda > 0), \\ 0 & (\lambda = 0). \end{cases}
\]
\end{lemma}
\begin{proof}
Let $\chi = \psi+\psi'$ be as in Lemma 1. Then $\widetilde{\chi} = \widetilde{\psi} + \widetilde{\psi'}$ and moreover $s(\widetilde{\psi}) = s(\psi)$, $s(\widetilde{\psi'}) = s(\psi')$ by \cite[Lemma 1(2)(c)]{Terp:thesis81}. By Lemma \ref{L1} we observe that 
\begin{align*}
[D\psi:D\varphi]_t 
&= 
s(\psi)[D\chi:D\varphi]_t, \\
[D\widetilde{\psi}:D\widetilde{\varphi}]_t 
&= 
s(\psi)[D\widetilde{\chi}:D\widetilde{\varphi}]_t, \\ 
[D\widetilde{\psi}:D\tau_M]_t 
&= s(\psi)[D\widetilde{\chi}:D\tau_M]_t 
\end{align*}
for every $t \in \mathbb{R}$. By the chain rule of Connes's Radon--Nikodym cocycles, we have 
\begin{align*} 
[D\widetilde{\psi}:D\tau_M]_t 
= 
s(\psi)[D\widetilde{\chi}:D\tau_M]_t 
&= 
s(\psi)[D\widetilde{\chi}:D\widetilde{\varphi}]_t\,
[D\widetilde{\varphi}:D\tau_M]_t \\
&= 
[D\widetilde{\psi}:D\widetilde{\varphi}]_t\,\lambda^{\varphi}(t) 
\end{align*}
holds for every $t \in \mathbb{R}$. By \cite[Theorem 11.9]{Stratila:Book81} we observe 
\[
[D\widetilde{\psi}:D\widetilde{\varphi}]_t 
= 
s(\psi)[D\widetilde{\chi}:D\widetilde{\varphi}]_t 
= 
s(\psi)[D\chi:D\varphi]_t 
= 
[D\psi:D\varphi]_t
\]
for every $t \in \mathbb{R}$. Consequently, we have 
\[
[D\widetilde{\psi}:D\tau_M]_t 
= 
[D\psi:D\varphi]_t \lambda^{\varphi}(t)
\]
for every $t \in \mathbb{R}$. 

The $h_\psi$ is defined to be the Radon--Nikodym derivative of $\widetilde{\psi}$ with respect to the canonical trace $\tau_M$, that is, $\widetilde{\psi} = \tau_M(h_\psi\,\cdot\,)$ in the sense of \cite[Lemma 2]{Terp:thesis81} or \cite[\S4.4]{Stratila:Book81}. By \cite[Corollary 4.8]{Stratila:Book81} we have $[D\widetilde{\psi}:D\tau_M]_t = h_\psi^{it}$. Hence we have equation \eqref{Eq6}, and it is obvious that equation \eqref{Eq6} characterizes $h_\varphi$ thanks to Stone's theorem. 
\end{proof} 

The Haagerup non-commutative $L^p$-space $L^p(M)$, $0 < p \leq \infty$,  is defined to be all $\tau_M$-measurable operators $h$ affiliated with $\widetilde{M}$ such that $\theta^M_t(h) = e^{-t/p}h$ for all $t \in \mathbb{R}$. The details on the space are referred to \cite{Terp:thesis81}. 

Here is another lemma, which is probably a known fact, but we give its proof for the sake of completeness. 

\begin{lemma}\label{L3}  
Assume that $p \geq 1$ and $\varphi$ is a faithful normal positive linear functional so that $M$ must be $\sigma$-finite. Let $\mathfrak{A} \subset M$ be a $\sigma$-weakly dense $*$-subalgebra. Then $\mathfrak{A} h_{\varphi}^{1/p}$ is dense in $L^p(M)$. 
\end{lemma}  
\begin{proof} 
We will use the (left) Kosaki non-commutative $L^p$-space $L^p(M,\varphi)$ with norm $\Vert\,\cdot\,\Vert_{p,\varphi}$, which is the complex interpolation space $C_{1/p}(M h_{\varphi},L^1(M))$, where the embedding $a \in M \mapsto ah_{\varphi} \in L^1(M)$ gives a compatible pair with norm $ah_{\varphi} \in M h_{\varphi} \mapsto \Vert ah_\varphi \Vert_\infty := \Vert a\Vert_M$ (operator norm) for $a \in M$. By \cite[Theorem 9.1]{Kosaki:JFA84-1} we have $L^p(M,\varphi) = L^p(M)h_{\varphi}^{1/q} \subset L^1(M)$ with $1/p + 1/q = 1$. 

For a given $a \in M$ the Kaplansky density theorem enables us to choose a net $a_\lambda \in \mathfrak{A}$ in such a way that $\Vert a_\lambda\Vert_M \leq \Vert a\Vert_M$ for all $\lambda$ and $a_\lambda \to a$ in the $\sigma$-weak topology. By means of complex interpolation theory, we obtain that 
\begin{align*} 
\Vert a_\lambda h_{\varphi} - a h_{\varphi} \Vert_{p,\varphi} 
&= 
\Vert (a_\lambda - a) h_{\varphi} \Vert_{p,\varphi} 
\leq 
\Vert (a_\lambda - a)h_{\varphi}\Vert_\infty^{1/q}\Vert(a_\lambda-a)h_{\varphi}\Vert_1^{1/p} \\
&= 
\Vert a_\lambda - a\Vert_M^{1/q} \Vert h_{(a_\lambda-a)\varphi}\Vert_1^{1/p} 
= 
\Vert a_\lambda - a\Vert_M^{1/q} \Vert (a_\lambda-a)\varphi\Vert_{M_*}^{1/p} \\
&\leq 
(2\Vert a\Vert_M)^{1/q} \Vert\varphi\Vert^{1/2p} \varphi((a_\lambda-a)^*(a_\lambda-a))^{1/2p} \to 0. 
\end{align*} 
Therefore, $\mathfrak{A}h_{\varphi}$ is dense in $L^p(M,\varphi) = L^p(M)h_{\varphi}^{1/q}$, because so is $Mh_{\varphi}$ thanks to a general fact on complex interpolation spaces. Hence, for each $x\in L^p(M)$ there exists a sequence $a_n \in \mathfrak{A}$ so that $\Vert a_n h_\varphi - xh_\varphi^{1/q}\Vert_{p,\varphi} \to 0$ as $n\to\infty$. Since $a_n h_{\varphi}=(a_n h_{\varphi}^{1/p})h_{\varphi}^{1/q}$ and since \cite[equation (21)]{Kosaki:JFA84-1} (with $\eta=0$ there), we conclude that $\Vert a_n h_\varphi^{1/p} - x\Vert_p \to 0$ as $n\to\infty$ so that $\mathfrak{A}h_{\varphi}^{1/p}$ is dense in $L^p(M)$.   
\end{proof}

\section{Main Results}
 
Let $M_i$, $i=1,2$, be von Neumann algebras. For each $i=1,2$, we choose a faithful semifinite normal weight $\varphi_i$ on $M_i$. Let 
\[
\widetilde{M}_i := M_i\bar{\rtimes}_{\sigma^{\varphi_i}}\mathbb{R}, \quad 
\widetilde{M_1\bar{\otimes}M_2} := (M_1\bar{\otimes}M_2)\bar{\rtimes}_{\sigma^{\varphi_1\bar{\otimes}\varphi_2}}\mathbb{R}
\]
be the continuous cores of $M_i$, $i=1,2$, and $M_1\bar{\otimes}M_2$ together with the dual actions $\theta^{(i)} := \theta^{M_i} : \mathbb{R} \curvearrowright \widetilde{M}_i$, $i=1,2$, and $\theta := \theta^{M_1\bar{\otimes}M_2} : \mathbb{R} \curvearrowright \widetilde{M_1\bar{\otimes}M_2}$.

The next fact is known in the structure analysis of type III factors. Especially, the fact is known among specialists on type III factors as a key tool to compute invariants such as flows of weights for tensor product type III factors. 

\begin{lemma}\label{L4} {\rm (Joint flow)} We have an identification 
\[
\widetilde{M_1\bar{\otimes}M_2} = (\widetilde{M}_1\bar{\otimes}\widetilde{M}_2)^{(\theta^{(1)}_{-t}\bar{\otimes}\theta_t^{(2)},\,\mathbb{R})} := \big\{x \in \widetilde{M}_1\bar{\otimes}\widetilde{M}_2\,;\, \theta^{(1)}_{-t}\bar{\otimes}\theta_t^{(2)}(x)=x \ \text{for all $t \in \mathbb{R}$}\big\}
\]
by 
\begin{align*}
\pi_{\varphi_1\bar{\otimes}\varphi_2}(a\otimes b) 
&= 
\pi_{\varphi_1}(a)\otimes\pi_{\varphi_2}(b), \quad a \in M_1, b \in M_2, \\ 
\lambda^{\varphi_1\bar{\otimes}\varphi_2}(t) 
&= 
\lambda^{\varphi_1}(t)\otimes\lambda^{\varphi_2}(t), \quad t \in \mathbb{R}. 
\end{align*}
Via this identification, 
\[
\theta_t 
= 
(\theta^{(1)}_t\bar{\otimes}\mathrm{id})\!\upharpoonright_{\widetilde{M_1\bar{\otimes}M_2}} 
= 
(\mathrm{id}\bar{\otimes}\theta_t^{(2)})\!\upharpoonright_{\widetilde{M_1\bar{\otimes}M_2}}, \qquad t \in \mathbb{R}.
\]
\end{lemma}
\begin{proof} This follows from the formula $\sigma_t^{\varphi_1\bar{\otimes}\varphi_2} = \sigma_t^{\varphi_1}\bar{\otimes}\sigma_t^{\varphi_2}$ and \cite[Theorem 21.8]{Stratila:Book81} that originates in \cite {Takesaki:ActaMath73}. Let us explain how to apply \cite[Theorem 21.8]{Stratila:Book81} to our problem. 

Let $G:=\mathbb{R}^2 > H:=\{(t,t); t\in\mathbb{R}\}$, a closed subgroup, and define $\sigma_g :=\sigma_{t_1}^{\varphi_1}\bar{\otimes}\sigma_{t_2}^{\varphi_2}$ for $g=(t_1,t_2)\in G$. Then we have an action $\sigma : G \curvearrowright M_1\bar{\otimes}M_2$, and its restrication to $H$ is the modular action $\sigma_t^{\varphi_1\bar{\otimes}\varphi_2}=\sigma_t^{\varphi_1}\bar{\otimes}\sigma_t^{\varphi_2}$. Thus, we have 
\[
\widetilde{M}_1\bar{\otimes}\widetilde{M}_2 = (M_1\bar{\otimes}M_2)\bar{\rtimes}_\sigma G \supset (M_1\bar{\otimes}M_2)\bar{\rtimes}_\sigma H = \widetilde{M_1\bar{\otimes}M_2},
\]
where
\[
(\pi_{\varphi_1}(a)\otimes\pi_{\varphi_2}(b))(\lambda^{\varphi_1}(t_1)\otimes\lambda^{\varphi_2}(t_2))=\pi_\sigma(a_1\otimes a_2)\lambda^\sigma(t_1,t_2), \quad a \in M_1, b \in M_2, (t_1,t_2) \in G
\]
on the first identity, the second inclusion is the natural one, and 
\[
\pi_\sigma(a_1\otimes a_2)\lambda^\sigma(t,t)=\pi_{\varphi_1\bar{\otimes}\varphi_2}(a_1\otimes a_2)\lambda^{\varphi_1\bar{\otimes}\varphi_2}(t), \quad a \in M_1, b \in M_2, (t,t) \in H
\]
on the third identity. Here, $\pi_\sigma : M_1\bar{\otimes}M_2 \to (M_1\bar{\otimes}M_2)\bar{\rtimes}_\sigma G$ and $\lambda^\sigma : G \to (M_1\bar{\otimes}M_2)\bar{\rtimes}_\sigma G$ denote the canonical injective normal $*$-homomorphism and the canonical unitary representation, respectively. We have $\widehat{G}=G$ with the dual pairing $\langle(t_1,t_2),(t'_1,t'_2)\rangle := e^{i(t_1 t'_1 + t_2 t'_2)}$ between $G$ and its copy, and $\widehat{H}$ becomes $\{(-t,t); t \in \mathbb{R}\}$ in $G$. Moreover, the dual action $\hat{\sigma}_g$ with $g=(t_1,t_2)\in G$ is given by $\theta_{t_1}^{(1)}\bar{\otimes}\theta_{t_2}^{(2)}$ via the above identification. Hence, the desired first assertion immediately follows by \cite[Theorem 21.8]{Stratila:Book81}. Then the desired identity of the dual action $\theta_t$ can easily be confirmed by investigating its behavior on the canonical generators. 
\end{proof}

In what follows, we use the description of the continuous core of $M_1\bar{\otimes}M_2$ equipped with the dual action $\theta$ in Lemma \ref{L4}. 

Remark that $\tau_{M_1\bar{\otimes}M_2}$ cannot be identified with a restriction of the tensor product trace $\tau_{M_1}\bar{\otimes}\tau_{M_2}$. However, $\tau_{M_1\bar{\otimes}M_2}$ is characterized by 
\begin{equation}\label{Eq7}
\big[D\widetilde{\varphi_1\bar{\otimes}\varphi_2}:D\tau_{M_1\bar{\otimes}M_2}\big]_t = \lambda^{\varphi_1}(t)\otimes\lambda^{\varphi_2}(t), \qquad t \in \mathbb{R}
\end{equation} 
in the description of Lemma \ref{L4}. This is indeed a key fact in the discussion below. 

\medskip
Fix a $p \in (0,\infty]$. Choose a pair $(x_1,x_2) \in L^p(M_1) \times L^p(M_2)$, whose entries can be regarded as unbounded operators on Hilbert spaces $\mathcal{H}_i$, $i=1,2$, on which $\widetilde{M}_i$ are constructed. Let $x_i = v_i|x_i|$, $i=1,2$, be their polar decompositions. Then $v_i \in M_i$ and $|x_i|^p \in L^1(M_i)$ for each $i=1,2$. Then, for each $i=1,2$, there is a  unique $\psi_i \in (M_i)_*^+$ so that $h_{\psi_i} = |x_i|^p$ holds. Here is a lemma. 

\begin{lemma}\label{L5} The following hold true: 
\begin{itemize}
\item[(1)] $x_1\bar{\otimes}x_2 = (v_1\otimes v_2)(|x_1|\bar{\otimes}|x_2|)$ is the polar decomposition, where the tensor product of $\tau$-measurable operators is understood as that on $\mathcal{H}_1\bar{\otimes}\mathcal{H}_2$. 
\item[(2)] $|x_1\bar{\otimes}x_2|^p = |x_1|^p\bar{\otimes}|x_2|^p$. 
\item[(3)] $h_{\psi_1}\bar{\otimes}h_{\psi_2} = h_{\psi_1\bar{\otimes}\psi_2}$ and $(h_{\psi_1}\bar{\otimes}h_{\psi_2})^{it} = h_{\psi_1\bar{\otimes}\psi_2}^{it}$. 
\end{itemize}
\end{lemma} 
\begin{proof}
Items (1),(2) can easily be confirmed within theory of unbounded operators. See Appendix A. 

Item (3): Observe that 
\begin{align*}
h_{\psi_1\bar{\otimes}\psi_2}^{it} 
&= 
[D\psi_1\bar{\otimes}\psi_2:D\varphi_1\bar{\otimes}\varphi_2]_t\,\lambda^{\varphi_1}(t)\otimes\lambda^{\varphi_2}(t) \\
&= 
([D\psi_1:D\varphi_2]_t\otimes [D\psi_2:D\varphi_2]_t)\,(\lambda^{\varphi_1}(t)\otimes\lambda^{\varphi_2}(t)) \\
&= 
 ([D\psi_1:D\varphi_1]_t\,\lambda^{\varphi_1}(t))\otimes([D\psi_2:D\varphi_2]_t\,\lambda^{\varphi_2}(t)) \\
 &= 
 h_{\psi_1}^{it}\otimes h_{\psi_2}^{it}   
\end{align*}
by equations \eqref{Eq6},\eqref{Eq7} and \cite[Corollary 8.6]{Stratila:Book81}. 
Since $(h_{\psi_1}\bar{\otimes}h_{\psi_2})^{it} = h_{\psi_1}^{it}\otimes h_{\psi_2}^{it}$ (see Appendix A), we obtain item (3) by Lemma \ref{L2} (its uniqueness part). 
\end{proof} 

By Lemma \ref{L5} we have 
\[
|x_1\bar{\otimes}x_2|^p = |x_1|^p\bar{\otimes}|x_2|^p = h_{\psi_1}\bar{\otimes}h_{\psi_2} = h_{\psi_1\bar{\otimes}\psi_2}. 
\]
Since $\psi_1\bar{\otimes}\psi_2 \in (M_1\bar{\otimes}M_2)_*$, we have $x_1\bar{\otimes}x_2 \in L^p(M_1\bar{\otimes}M_2)$ and 
\begin{align*}
\Vert x_1\bar{\otimes}x_2\Vert_p^p 
&= 
\mathrm{tr}(|x_1\bar{\otimes}x_2|^p) 
=
\mathrm{tr}(|x_1|^p\bar{\otimes}|x_2|^p) 
=
\mathrm{tr}(h_{\psi_1}\bar{\otimes}h_{\psi_2}) 
=
\mathrm{tr}(h_{\psi_1\bar{\otimes}\psi_2}) \\
&= 
(\psi_1\bar{\otimes}\psi_2)(1) = \psi_1(1)\psi_2(1) 
= 
\mathrm{tr}(h_{\psi_1})\mathrm{tr}(h_{\psi_2}) 
= 
\mathrm{tr}(|x_1|^p)\mathrm{tr}(|x_2|^p) 
= 
\Vert x_1\Vert_p^p\,\Vert x_2\Vert_p^p. 
\end{align*}
Consequently, we have the first part of the following theorem: 

\begin{theorem}\label{T6} For any pair $(x_1,x_2) \in L^p(M_1) \times L^p(M_2)$ the unbounded operator tensor product $x_1\bar{\otimes}x_2$ affiliated with $\widetilde{M}_1\bar{\otimes}\widetilde{M}_2$ actually gives an element of $L^p(M_1\bar{\otimes}M_2)$, and then 
\[
\Vert x_1\bar{\otimes}x_2\Vert_p = \Vert x_1\Vert_p\,\Vert x_2\Vert_p
\] 
holds. 

The mapping $(x_1,x_2) \mapsto x_1\bar{\otimes}x_2$ is clearly bilinear, and induces a natural map from the vector space tensor product $L^p(M_1)\otimes_\mathrm{alg}L^p(M_2)$ into $L^p(M_1\bar{\otimes}M_2)$, which has dense image when both $M_i$ are $\sigma$-finite and $p \geq 1$. 
\end{theorem} 
\begin{proof}  
Let us prove the second part. We can assume that both $\varphi_i$ are faithful normal states. Then $(M\otimes_\mathrm{alg}N)h_{\varphi_1\bar{\otimes}\varphi_2}^{1/p}$ is dense in $L^p(M\bar{\otimes}N)$ by Lemma \ref{L3}. Therefore, the $(ah_{\varphi_1}^{1/p})\bar{\otimes}(bh_{\varphi_2}^{1/p}) = (a\otimes b)\,h_{\varphi_1}^{1/p}\bar{\otimes}h_{\varphi_2}^{1/p} = (a\otimes b)h_{\varphi_1\bar{\otimes}\varphi_2}^{1/p} $ with $a \in M_1$ and $b \in M_2$ are total in $L^p(M_1\bar{\otimes}M_2)$. Hence we are done. 
\end{proof} 

We do not know whether or not the second assertion (the density of the induced map) in the above theorem holds without $\sigma$-finiteness. However, we think that an approximation by $\sigma$-finite projections might give the same assertion without $\sigma$-finiteness. We leave this question to the interested reader. 

\medskip
Here is a corollary on the Kosaki non-commutative $L^p$-space $L^p(M,\varphi)_\eta$ with $1 \leq p \leq \infty$ and $0 \leq \eta \leq 1$, which is defined as the complex interpolation space $C_{1/p}(h_\varphi^\eta M h_\varphi^{1-\eta},L^1(M))$, where the embedding $a \in M \mapsto h_\varphi^\eta ah_\varphi^{1-\eta} \in L^1(M)$ gives a compatible pair with norm $h_\varphi^\eta ah_\varphi^{1-\eta} \in h_\varphi^\eta M h_\varphi^{1-\eta}\mapsto \Vert h_\varphi^\eta a h_\varphi^{1-\eta}\Vert_\infty:=\Vert a\Vert_M$.   

\begin{corollary}\label{C7} Assume that both $M_i$ are $\sigma$-finite, and both $\varphi_i$ are faithful normal positive linear functionals. For each $1 \leq p \leq \infty$ and $0\leq\eta\leq1$, the mapping $(x_1,x_2) \in L^1(M_1)\times L^1(M_2) \mapsto x_1\bar{\otimes}x_2 \in L^1(M_1\bar{\otimes}M_2)$ in Theorem \ref{T6} induces a bilinear map from the vector space tensor product $L^p(M_1,\varphi_1)_\eta\otimes_\mathrm{alg} L^p(M_2,\varphi_2)_\eta$ into $L^p(M_1\bar{\otimes}M_2,\varphi_1\bar{\otimes}\varphi_2)_\eta$ with dense image, and then 
\[
\Vert x_1\bar{\otimes}x_2\Vert_{p,\varphi_1\bar{\otimes}\varphi_2,\eta} = \Vert x_1\Vert_{p,\varphi_1,\eta}\Vert x_2\Vert_{p,\varphi_2,\eta}
\] 
holds. 
\end{corollary}
\begin{proof} The Kosaki non-commutative $L^p$-spaces $L^p(M_i,\varphi_i)_\eta$ and $L^p(M_1\bar{\otimes}M_2,\varphi_1\bar{\otimes}\varphi_2)_\eta$ are
\begin{gather*}
h_{\varphi_i}^{\eta/q}L^p(M_i)h_{\varphi_i}^{(1-\eta)/q} \subset L^1(M_i), \\
h_{\varphi_1\bar{\otimes}\varphi_2}^{\eta/q}L^p(M_1\bar{\otimes}M_2)h_{\varphi_1\bar{\otimes}\varphi_2}^{(1-\eta)/q} \subset L^1(M_1\bar{\otimes}M_2),
\end{gather*}
respectively, where $q$ is the dual exponent of $p$, that is, $1/p+1/q=1$. 

Each $(x_1,x_2) \in L^p(M_1,\varphi_1)_\eta\times L^p(M_2,\varphi_2)_\eta \subset L^1(M_1)\times L^1(M_2)$ is of the form 
\[
(h_{\varphi_1}^{\eta/q}x_1' h_{\varphi_1}^{(1-\eta)/q},h_{\varphi_1}^{\eta/q}x_2'h_{\varphi_2}^{(1-\eta)/q})
\] 
with unique $x_1' \in L^p(M_1)$ and $x_2' \in L^p(M_2)$. Then, we have, by Lemma \ref{L-A1} and Lemma \ref{L5}(2)(3),
\begin{align*}
x_1\bar{\otimes}x_2 
= 
(h_{\varphi_1}^{\eta/q}\bar{\otimes}h_{\varphi_2}^{\eta/q})(x_1'\bar{\otimes}x_2')(h_{\varphi_1}^{(1-\eta)/q}\bar{\otimes}h_{\varphi_2}^{(1-\eta)/q})
= 
h_{\varphi_1\bar{\otimes}\varphi_2}^{\eta/q}(x_1'\bar{\otimes}x_2')h_{\varphi_1\bar{\otimes}\varphi_2}^{(1-\eta)/q}, 
\end{align*}
which clearly falls in $L^p(M_1\bar{\otimes}M_2,\varphi_1\bar{\otimes}\varphi_2)_\eta$ since $x'_1\bar{\otimes}x'_2 \in L^p(M_1\bar{\otimes}M_2)$ by Theorem \ref{T6}. 

Moreover, we observe that
\begin{align*}
\Vert x_1\bar{\otimes}x_2\Vert_{p,\varphi_1\bar{\otimes}\varphi_2,\eta} 
&=
\Vert h_{\varphi_1\bar{\otimes}\varphi_2}^{\eta/q}(x_1'\bar{\otimes}x_2')h_{\varphi_1\bar{\otimes}\varphi_2}^{(1-\eta)/q}\Vert_{p,\varphi_1\bar{\otimes}\varphi_2,\eta} \\
&=
\Vert x_1'\bar{\otimes}x_2'\Vert_p 
=
\Vert x_1'\Vert_p \Vert x_2'\Vert_p \quad \text{(by Theorem \ref{T6} again)} \\
&=
\Vert h_{\varphi_1}^{\eta/q}x_1' h_{\varphi_1}^{(1-\eta)/q}\Vert_{p,\varphi_1,\eta} \Vert h_{\varphi_1}^{\eta/q}x_2' h_{\varphi_2}^{(1-\eta)/q}\Vert_{p,\varphi_2,\eta} \\
&= 
\Vert x_1\Vert_{p,\varphi_1,\eta} \Vert x_2\Vert_{p,\varphi_2,\eta}.
\end{align*}

That the map from $L^p(M_1,\varphi_1)_\eta\otimes_\mathrm{alg} L^p(M_2,\varphi_2)_\eta$ into $L^p(M_1\bar{\otimes}M_1,\varphi_1\bar{\otimes}\varphi_2)_\eta$ has dense image follows from Theorem \ref{T6} together with the definition of norm $\Vert\,\cdot\,\Vert_{p,\varphi_1\bar{\otimes}\varphi_2,\eta}$.  
\end{proof}

Here is a question. Let $(x_1,x_2) \in L^1(M_1,\varphi_1)_\eta\times L^1(M_2,\varphi_2)_\eta$ be arbitrarily given with $x_i\neq0$. Does $x_1\bar{\otimes}x_2 \in L^p(M_1\bar{\otimes}M_2,\varphi_1\bar{\otimes}\varphi_2)_\eta$ imply that $x_i \in L^p(M_i,\varphi_i)_\eta$ for both $i=1,2$ ?

\section{A Sample of Application in QIT} 

Here we illustrate how our description of non-commutative $L^p$-spaces $L^p(M_1\bar{\otimes}M_2)$ is useful.  

Let $\alpha \in [1/2,\infty)\setminus\{1\}$ be given. The sandwiched $\alpha$-R\'{e}nyi divergence $\widetilde{D}_\alpha(\psi || \varphi)$ for  $\psi, \varphi \in M_*^+$ with $\psi\neq 0$ allows several definitions, one of which is  
\[
\widetilde{D}_\alpha(\psi||\varphi) := \frac{1}{\alpha-1}\log\frac{\widetilde{Q}_\alpha(\psi||\varphi)}{\psi(1)}, 
\]
where 
\[
\widetilde{Q}_\alpha(\psi||\varphi) := 
\begin{cases} 
\mathrm{tr}\big[\big(h_\varphi^{(1-\alpha)/2\alpha}h_\psi h_\varphi^{(1-\alpha)/2\alpha}\big)^\alpha\big] & (\text{$1/2 \leq \alpha<1$}), \\
\Vert h_\psi \Vert_{\alpha,\varphi,1/2}^\alpha & (\text{$\alpha>1$, $s(\psi) \leq s(\varphi)$ and $h_\psi \in L^\alpha(M,\varphi)_{1/2}$}), \\
+ \infty & (\text{otherwise}).
\end{cases} 
\]
This formulation is mainly due to Jen\v{c}ov\'{a}. See \cite[3.3]{Hiai:Book21}. 

The sandwiched $\alpha$-R\'{e}nyi divergence $\widetilde{Q}_\alpha(\psi||\varphi)$  admits a two parameter extension, called the \emph{$\alpha$-$z$-R\'{e}nyi divergence}, in the finite-dimensional or more generally the infinite-dimensional type I setup. See \cite{JaksicOgataPautratPillet:LectNotes12},\cite{AudenaertDatta:JMP15},\cite{Mosonyi:CMP23} in historical order. Here we propose a possible definition of its non-type I extension, for we want to explain how the present description of non-commutative $L^p$-spaces associated with tensor product von Neumann algebras works even for the extension. Let $\alpha, z > 0$ with $\alpha\neq1$ be arbitrarily given. For each pair $\varphi, \psi \in M_*^+$ with $\psi\neq0$ we define $\widetilde{Q}_{\alpha,z}(\psi||\varphi)$ to be
\[
\begin{cases} 
\mathrm{tr}\big[\big(h_\varphi^{(1-\alpha)/2z}h_\psi^{\alpha/z} h_\varphi^{(1-\alpha)/2z}\big)^z\big] & (\text{$\alpha<1$}), \\
\Vert x \Vert_z^z & (\text{$\alpha>1$ and ($\spadesuit$) holds with $x \in s(\varphi)L^z(M)s(\varphi)$}), \\
+ \infty & (\text{otherwise}), 
\end{cases}
\]
where
\[
\text{($\spadesuit$)} \qquad h_\psi^{\alpha/z} = h_\varphi^{(\alpha-1)/2z} x h_\varphi^{(\alpha-1)/2z}. 
\]

\begin{lemma}\label{L8}
Identity {\rm ($\spadesuit$)} uniquely determines $x \in s(\varphi)L^z(M)s(\varphi)$ {\rm (}if it exists{\rm)}. Hence the above definition of $\widetilde{Q}_{\alpha,z}(\psi||\varphi)$ is well defined. 
\end{lemma}
\begin{proof}

Assume that $y \in s(\varphi)L^z(M)s(\varphi)$ also satisfies $h_\psi^{\alpha/z} = h_\varphi^{(\alpha-1)/2z} y h_\varphi^{(\alpha-1)/2z}$. 
Since all the $\tau$-measurable operators form a $*$-algebra, we have
\begin{equation*}
0 = h_\varphi^{(\alpha-1)/2z} x h_\varphi^{(\alpha-1)/2z} - 
h_\varphi^{(\alpha-1)/2z} y h_\varphi^{(\alpha-1)/2z}
= h_\varphi^{(\alpha-1)/2z} (x - y) h_\varphi^{(\alpha-1)/2z}. 
\end{equation*}
Moreover, $h_\varphi$ is a $\tau$-measurable operator and non-singular affiliated with 
$s(\varphi) \widetilde{M} s(\varphi)$. 
In addition, $f(t) = t^{(\alpha-1)/2z}$ is a continuous strictly monotone increasing  function 
on $[0, \infty)$ with $f(0) = 0$ if $\alpha > 1$.
Hence, $h_\varphi^{(\alpha-1)/2z}$ is also $\tau$-measurable (see e.g., \cite[Proposition 4.19]{Hiai:Book21-2}; but this fact is implicitly utilized in the general theory of Haagerup non-commutative $L^p$-spaces that we have employed) and non-singular. Thus, $x - y = 0$ by \cite[Lemma 2.1]{Kosaki:MathScand81}, and hence $x=y$. 
\end{proof}

Then 
\[
\widetilde{D}_{\alpha,z}(\psi||\varphi) := 
\frac{1}{\alpha-1}\log\frac{\widetilde{Q}_{\alpha,z}(\psi||\varphi)}{\psi(1)}
\]
should be called the \emph{$\alpha$-$z$-R\'{e}nyi divergence}. 

\begin{lemma}\label{L9} $\widetilde{Q}_{\alpha,\alpha}(\psi||\varphi) = \widetilde{Q}_\alpha(\psi||\varphi)$ holds for every $\alpha \geq 1/2$ with $\alpha \neq 1$.
\end{lemma}
\begin{proof}
When $\alpha \in [1/2, 1)$ the desired identity clearly holds by the definitions of 
$\widetilde{Q}_{\alpha,\alpha}(\psi||\varphi)$ and $\widetilde{Q}_\alpha(\psi||\varphi)$. 

We then consider the case when $\alpha = z > 1$. Then,  identity ($\spadesuit$) holds with $x \in s(\varphi)L^\alpha(M)s(\varphi)$ if and only if $h_\psi \in h_\varphi^{(\alpha-1)/2\alpha} L^\alpha(M) h_\varphi^{(\alpha-1)/2\alpha} = L^\alpha(M,\varphi)_{1/2}$, where the dual exponent of $\alpha$ is $\alpha/(\alpha - 1)$. Moreover, in this case, we have $\Vert h_\psi \Vert_{\alpha,\varphi,1/2}^\alpha = \Vert x \Vert_\alpha^\alpha$. 
\end{proof} 

The next fact was claimed for the sandwiched $\alpha$-R\'{e}nyi divergence $\widetilde{Q}_\alpha(\psi||\varphi)$ in \cite[page 1860]{BertaScholzTomamichel:AHP18} without detailed proof. Then, its detailed proof when both $M_i$ are injective or AFD was given by Hiai and Mosonyi \cite[equation (3.16)]{HiaiMosonyi:AHP23} by using the finite-dimensional result and also the martingale convergence property that they established. We believe that the proof below is more natural than those. 

\begin{proposition}\label{P11} For any $\psi_i,\varphi_i \in (M_i)_*^+$ with $\psi_i\neq0$, $i=1,2$, we have 
\[
\widetilde{Q}_{\alpha,z}(\psi_1\bar{\otimes}\psi_2||\varphi_1\bar{\otimes}\varphi_2) = \widetilde{Q}_{\alpha,z}(\psi_1||\varphi_1) \widetilde{Q}_{\alpha,z}(\psi_2||\varphi_2) 
\]
and 
\[
\widetilde{D}_{\alpha,z}(\psi_1\bar{\otimes}\psi_2||\varphi_1\bar{\otimes}\varphi_2) = \widetilde{D}_{\alpha,z}(\psi_1||\varphi_1) + \widetilde{D}_{\alpha,z}(\psi_2||\varphi_2), 
\]
when $\alpha < 1$ or both $\widetilde{Q}_{\alpha,z}(\psi_i||\varphi_i) < +\infty$. 

When $\alpha=z \in [1/2,1)\cup(1,\infty)$, the identities hold without the above assumption that $\alpha < 1$ or both $\widetilde{Q}_{\alpha,z}(\psi_i||\varphi_i) < +\infty$. 
\end{proposition} 
\begin{proof} 
We first consider the case when $\alpha < 1$. We have, by Lemma \ref{L5}(2)(3) together with Lemma \ref{L-A1},  
\begin{align*}
(h_{\varphi_1\bar{\otimes}\varphi_2}^{(1-\alpha)/2z}\,h_{\psi_1\bar{\otimes}\psi_2}^{\alpha/z}\,h_{\varphi_1\bar{\otimes}\varphi_2}^{(1-\alpha)/2z})^z 
&= 
((h_{\varphi_1}^{(1-\alpha)/2z}\bar{\otimes}h_{\varphi_2}^{(1-\alpha)/2z})(h_{\psi_1}^{\alpha/z}\bar{\otimes}h_{\psi_2} ^{\alpha/z})(h_{\varphi_1}^{(1-\alpha)/2z}\bar{\otimes}h_{\varphi_2}^{(1-\alpha)/2z}))^z \\
&= 
((h_{\varphi_1}^{(1-\alpha)/2z} h_{\psi_1}^{\alpha/z} h_{\varphi_1}^{(1-\alpha)/2z})\bar{\otimes}(h_{\varphi_2}^{(1-\alpha)/2z} h_{\psi_2}^{\alpha/z} h_{\varphi_2}^{(1-\alpha)/2z}))^z \\
&= 
(h_{\varphi_1}^{(1-\alpha)/2z} h_{\psi_1}^{\alpha/z} h_{\varphi_1}^{(1-\alpha)/2z})^z\bar{\otimes}(h_{\varphi_2}^{(1-\alpha)/2z} h_{\psi_2}^{\alpha/z} h_{\varphi_2}^{(1-\alpha)/2z})^z
\end{align*}
as unbounded operators on $\mathcal{H}_1\bar{\otimes}\mathcal{H}_2$, on which $\widetilde{M}_1\bar{\otimes}\widetilde{M}_2$ naturally act. Since both the tensor components of the above right-most side fall into $L^1(M_i)$, $i=1,2$, respectively, we conclude, by Theorem \ref{T6}, that the desired multiplicativity of $\widetilde{Q}_{\alpha,z}$ holds true.  

\medskip 
We then consider the case of $\alpha>1$. Assume first that $\widetilde{Q}_{\alpha,z}(\psi_i||\varphi_i)<+\infty$ for both $i=1,2$. 
Then there are $x_i \in s(\varphi_i)L^z(M_i)s(\varphi_i)$, $i=1,2$, such that
$h_{\psi_i}^{\alpha/z} = h_{\varphi_i}^{(\alpha-1)/2z} x_i h_{\varphi_i}^{(\alpha-1)/2z}$. By Lemma \ref{L5}(2)(3) and Lemma \ref{L-A1}, we have
\begin{align*}
h_{\varphi_1\bar{\otimes}\varphi_2}^{(\alpha-1)/2z}\,x_1 \bar{\otimes}x_2\,h_{\varphi_1\bar{\otimes}\varphi_2}^{(\alpha-1)/2z}
&= (h_{\varphi_1}^{(\alpha-1)/2z} \bar{\otimes} h_{\varphi_2}^{(\alpha-1)/2z}) 
(x_1\bar{\otimes}x_2)
(h_{\varphi_1}^{(\alpha-1)/2z} \bar{\otimes} h_{\varphi_2}^{(\alpha-1)/2z}) \\
&= (h_{\varphi_1}^{(\alpha-1)/2z} x_1 h_{\varphi_1}^{(\alpha-1)/2z}) \bar{\otimes}
 (h_{\varphi_2}^{(\alpha-1)/2z} x_2 h_{\varphi_2}^{(\alpha-1)/2z}) \\
&= h_{\psi_1}^{\alpha/z} \bar{\otimes} h_{\psi_2}^{\alpha/z} 
= h_{\psi_1\bar{\otimes}\psi_2}^{\alpha/z}.
\end{align*}
Therefore, $\widetilde{Q}_{\alpha,z}(\psi_1\bar{\otimes}\psi_2||\varphi_1\bar{\otimes}\varphi_2) = \Vert x_1 \bar{\otimes} x_2 \Vert_z^z = \Vert x_1 \Vert_z^z \Vert x_2 \Vert_z^z = \widetilde{Q}_{\alpha,z}(\psi_1||\varphi_1) \widetilde{Q}_{\alpha,z}(\psi_2||\varphi_2)$ holds by Theorem \ref{T6}. 


\medskip
We finally assume that $\alpha=z>1$. By Lemma \ref{L9}, $\widetilde{Q}_{\alpha,z}(\psi_1\bar{\otimes}\psi_2||\varphi_1\bar{\otimes}\varphi_2)=\widetilde{Q}_\alpha(\psi_1\bar{\otimes}\psi_2||\varphi_1\bar{\otimes}\varphi_2)$ and $\widetilde{Q}_{\alpha,z}(\psi_i||\varphi_i)=\widetilde{Q}_\alpha(\psi_i||\varphi_i)$, $i=1,2$, hold. We also assume that at least one of the $\widetilde{Q}_\alpha(\psi_i||\varphi_i)$ is infinite. We may and do assume that $\widetilde{Q}_\alpha(\psi_1||\varphi_1)=+\infty$. Then, we apply the monotonicity property (see \cite[Theorem 3.16(4)]{Hiai:Book21}) to the unital $*$-homomorphism (i.e., unital normal positive map) $\gamma : M_1 \to M_1\bar{\otimes}M_2$ sending $a\in M_1$ to $a\otimes1$ and obtain that 
\[
\widetilde{Q}_\alpha(\psi_1\bar{\otimes}\psi_2||\varphi_1\bar{\otimes}\varphi_2)\geq \widetilde{Q}_\alpha(\psi_2(1)\psi_1||\varphi_2(1)\varphi_1). 
\]
If $\varphi_2(1)=0$, then $\varphi_2(1)\varphi_1=0$ and thus $\widetilde{Q}_\alpha(\psi_2(1)\psi_1||\varphi_2(1)\varphi_1)=+\infty$. Hence we may and do assume that $\varphi_2(1)\neq0$. Then, $s(\psi_2(1)\psi_1)\leq s(\varphi_2(1)\varphi_1)$ if and only if $s(\psi_1) \leq s(\varphi_1)$, and we easily confirm that $h_{\psi_2(1)\psi_1} = \psi_2(1)h_{\psi_1} \in L^\alpha(M_1,\varphi_2(1)\varphi_1)_{1/2}$ if and only if $h_{\psi_1} \in L^\alpha(M_1,\varphi_1)_{1/2}$. Hence $\widetilde{Q}_\alpha(\psi_2(1)\psi_1||\varphi_2(1)\varphi_1) = +\infty$, and then the desired multiplicative property of $\widetilde{Q}_\alpha$ holds as $+\infty=+\infty$.
\end{proof}

Proving the above additivity property in full generality, i.e., without assuming $\widetilde{Q}_{\alpha,z}(\psi_i||\varphi_i) < +\infty$, needs the monotonicity property (the so-called data processing inequality) or such a kind of property. It is an important question to determine the range of $(\alpha,z)$, for which the monotonicity property holds. (This question was completely settled by Zhang \cite{Zhang:AdvMath20} in the finite-dimensional case.) Moreover, general properties on $\widetilde{Q}_{\alpha,z}(\psi||\varphi)$ will be discussed in a future work \cite{Kato}.

\appendix

\section{Tensor products of ($\tau$-measurable) unbounded operators} 

Let $P$ and $Q$ be semifinite von Neumann algebras, which act on Hilbert spaces $\mathcal{H}$ and $\mathcal{K}$, respectively. Let $\tau_P$ and $\tau_Q$ be faithful semifinite normal tracial weights on $P$ and $Q$, respectively. We are investigating the tensor product of $\tau_P$-measurable $x$ and $\tau_Q$-measurable $y$. Recall that $x$ is said to be a $\tau_P$-measurable operator, if it is a closed densely defined operator affiliated with $P$ (denoted by $x\eta P$) such that for each $\delta>0$ there is a projection $e \in P$ such that $e\mathcal{H} \subset \mathrm{D}(x)$ and $\tau_P(1-e) < \delta$, where $\mathrm{D}(x)$ denotes the domain of $x$ as an operator on $\mathcal{H}$.  The same condition is applied to the $y$ too with replacing $(P,\tau_P)$ with $(Q,\tau_Q)$. 

By e.g., \cite[Lemma 7.21]{Schmudgen:Book} the algebraic tensor product of $x$ and $y$ becomes a densely defined closable operator, whose closure is denoted by $x\bar{\otimes}y$. It is also known, see e.g., \cite[Proposition 7.26]{Schmudgen:Book}, that $(x\bar{\otimes}y)^* = x^*\bar{\otimes}y^*$ holds true in general. Let $x = u|x|$ and $y=v|y|$ be the polar decompositions, see e.g., \cite[Theorem 7.2]{Schmudgen:Book}. Then, we have the polar decomposition  $x\bar{\otimes}y = (u\otimes v)(|x|\bar{\otimes}|y|)$, and in particular, $|x\bar{\otimes}y|=|x|\bar{\otimes}|y|$; see \cite[Exercise 7.6.13]{Schmudgen:Book}. This is indeed Lemma \ref{L5}(1). By e.g., \cite[Proposition 7.26]{Schmudgen:Book} again we observe that $|x|\bar{\otimes}|y|$ is self-adjoint. Moreover, by e.g., \cite[section 5.5.1 and Lemma 7.24]{Schmudgen:Book}, there is a unique spectral measure $e_{|x|,|y|}$ over $[0,\infty)^2$ such that $e_{|x|,|y|}(\Lambda_1\times\Lambda_2)=e_{|x|}(\Lambda_1)\otimes e_{|y|}(\Lambda_2)$ for any Borel subsets $\Lambda_1, \Lambda_2 \subset [0,\infty)$, and more importantly, 
\[
|x|\bar{\otimes}|y| = \int_{[0,\infty)^2} \lambda_1\lambda_2\,e_{|x|,|y|}(d\lambda_1,d\lambda_2)
\]
holds in the sense of spectral integrals. Here, $e_{|x|}$ and $e_{|y|}$ are the spectral measures of $|x|$ and $|y|$, respectively. Define 
\[
e_{|x|\bar{\otimes}|y|}(\Lambda) := \int_{[0,\infty)^2} \mathbf{1}_\Lambda(\lambda_1\lambda_2)\,e_{|x|,|y|}(d\lambda_1,d\lambda_2), \qquad \Lambda \subset [0,\infty). 
\]
Then, it is not hard to see that $e_{|x|\bar{\otimes}|y|}$ gives a unique spectral measure of $|x|\bar{\otimes}|y|$. Thus, we have
\[
f(|x|\bar{\otimes}|y|) 
= 
\int_{[0,\infty)}f(\lambda)\,e_{|x|\bar{\otimes}|y|}(d\lambda)
=
\int_{[0,\infty)^2}f(\lambda_1\lambda_2)\,e_{|x|,|y|}(d\lambda_1,d\lambda_2)
\]
for a non-negative Borel function $f$ on $[0,\infty)$. If $f(\lambda_1\lambda_2)=f(\lambda_1)f(\lambda_2)$, then 
\[
f(|x|\bar{\otimes}|y|)
=
\int_{[0,\infty)^2}f(\lambda_1)f(\lambda_2)\,e_{|x|,|y|}(d\lambda_1,d\lambda_2)
=
f(|x|)\bar{\otimes}f(|y|).
\]
This also holds even if the non-negativity of $f$ is replaced with the boundedness of $f$. Hence the above formula is applicable to the $f_t$ in Lemma \ref{L2}. Hence what we have established indeed includes Lemma \ref{L5}(2). The properties we have explained so far are valid for just closed densely defined (unbounded) operators, without $\tau$-measurability and even affiliation with $P$ and $Q$. 

Using the monotone class theorem in measure theory we can easily see that the spectral measure $e_{|x|,|y|}$ takes values in $P\bar{\otimes}Q$ since $e_{|x|,|y|}(\Lambda_1\times\Lambda_2)=e_{|x|}(\Lambda_1)\otimes e_{|y|}(\Lambda_2)\in P\bar{\otimes}Q$ by affiliation. Therefore, $|x\bar{\otimes}y|=|x|\bar{\otimes}|y|$ is affiliated with $P\bar{\otimes}Q$. It is non-trivial whether or not $x\bar{\otimes}y$, or equivalently, $|x|\bar{\otimes}|y|$, is $\tau_P\bar{\otimes}\tau_Q$-measurable. 
In fact, this is not the case in general; see \cite[Examples 3.14, 3.15]{AnoussisFelouzisTodorov-IJM15}. 

Let $x'$ and $y'$ be $\tau_P$- and $\tau_Q$-measurable operators, respectively. It is known that the usual products $xx'$ and $yy'$ are densely defined closable, and then the closures $\overline{xx'}$ and $\overline{yy'}$ become $\tau_P$- and $\tau_Q$-measurable again. By \cite[Lemma 7.22]{Schmudgen:Book} we have $(\overline{xx'})\bar{\otimes}(\overline{yy'}) = (xx')\bar{\otimes}(yy')$, whose right-hand side coincides with the closure of $(x\bar{\otimes} y)(x'\bar{\otimes} y')$. Thus, the following holds:

\begin{lemma}\label{L-A1} 
All the linear combinations of `simple tensors' $x\bar{\otimes}y$ with $\tau_P$- and $\tau_Q$-measurable $x\eta P$ and $y\eta Q$ form a $*$-algebra with strong sum and strong product. We simply understand $(x\bar{\otimes} y)(x'\bar{\otimes} y')$ as the strong product of $x\bar{\otimes} y$ and $x'\bar{\otimes} y'$ without the use of closure sign. With this notational rule, 
\[
(x\bar{\otimes} y)(x'\bar{\otimes} y')=(xx')\bar{\otimes}(yy')
\]
holds for any $\tau_P$- and $\tau_Q$-measurable $x,x'\eta P$ and $y,y' \eta Q$, where $xx'$ and $yy'$ are understood as the strong product. 
\end{lemma}

}

\end{document}